\documentclass[a4paper,11pt,reqno]{smfart}
\usepackage{amssymb}
\usepackage{etex}
\usepackage{amssymb,amsmath,mathrsfs,graphicx,float}
\usepackage[T1]{fontenc}
\usepackage[all]{xy}
\usepackage[english,francais]{babel}
\usepackage{mathptm}

\theoremstyle{plain}
\newtheorem{thm}{Theorem}[section]
\newtheorem{pro}[thm]{Proposition}
\newtheorem{lem}[thm]{Lemma}
\newtheorem{cor}[thm]{Corollary}

\newtheorem{theoalph}{Theorem}

\newtheorem{proalph}[theoalph]{Proposition}

\theoremstyle{definition}

\newtheorem{rem}[thm]{Remark}

\addtocounter{section}{0}             % Start with section 1
\numberwithin{equation}{section}       % Number formulas within sections

\usepackage[colorlinks, linktocpage, citecolor = blue, linkcolor = blue]{hyperref}

\begin{document}
\selectlanguage{english}

\title[The group generated by the automorphisms of $\mathbb{P}^n_\mathbb{C}$ and the standard involution]{Some properties of the group of birational maps generated by the automorphisms of $\mathbb{P}^n_\mathbb{C}$ \\
and the standard involution}

\thanks{The author is supported by ANR Grant "BirPol"  ANR-11-JS01-004-01.}

\author{Julie D\'eserti}
\address{Institut de Math\'ematiques de Jussieu-Paris Rive Gauche, UMR $7586$, Universit\'e 
Paris $7$, B\^atiment Sophie Germain, Case $7012$, $75205$ Paris Cedex 13, France.}
\email{deserti@math.univ-paris-diderot.fr}

\maketitle

\begin{abstract}
We give some properties of the subgroup $G_n(\mathbb{C})$ of the group of birational self-maps of~$\mathbb{P}^n_\mathbb{C}$ generated by the standard involution and the group of automorphisms of $\mathbb{P}^n_\mathbb{C}$. We prove that there is no nontrivial finite-dimensional linear representation of $G_n(\mathbb{C})$. We also establish that $G_n(\mathbb{C})$ is perfect, and that $G_n(\mathbb{C})$ equipped with the Zariski topology is simple. Furthermore if $\varphi$ is an automorphism of~$\mathrm{Bir}(\mathbb{P}^n_\mathbb{C})$, then up to birational conjugacy, and up to the action of a field automorphism $\varphi_{\vert G_n(\mathbb{C})}$ is trivial.

\end{abstract}

\section{Introduction}

The group $\mathrm{Bir}(\mathbb{P}^2_\mathbb{C})$ of birational self-maps of $\mathbb{P}^2_\mathbb{C}$, also called the Cremona group of rank~$2$, has been the object of a lot of studies. For finite subgroups let us mention for example \cite{Blanc:gpefini, DolgachevIskovskikh, BogomolovProkhorov}; other subgroups have been dealt with (\cite{Deserti:autbir, Deserti:nilpotent}), and some group properties have been established (\cite{Deserti:autbir, Deserti:hopfien, Deserti:sln, CerveauDeserti, Cantat:tits, CantatLamy, Blanc, Blanc:relation, Blanc:ssgpalg}). One can also find a lot of properties between algebraic geometry and dynamics (\cite{DillerFavre, CerveauDeserti:centralisateur, BlancDeserti:degreegrowth}). The Cremona group in higher dimension is far less well known; let us mention some references about finite subgroups (\cite{Prokhorov:simplefinitedim3, Prokhorov:invP3, Prokhorov:pelementarysubgoups, Popov3, Mundet}), about algebraic subgroups of maximal rank (\cite{Demazure, Vinberg, Umemura1, Umemura2, Umemura3}), about other subgroups (\cite{Popov1, Popov2}), about (abstract) homomorphisms from $\mathrm{PGL}(r+1;\mathbb{C})$ to the group $\mathrm{Bir}(M)$ where $M$ denotes a complex projective variety (\cite{Cantat}), and about maps of small bidegree (\cite{Pan, PanRongaVust, Pan:33map, Hudson, DesertiHan}). 

In this article we consider the subgroup of birational self-maps of $\mathbb{P}^n_\mathbb{C}$ introduced by Coble in~\cite{Coble}
\[
G_n(\mathbb{C})=\langle \sigma_n,\,\mathrm{Aut}(\mathbb{P}^n_\mathbb{C})\rangle
\]
where $\sigma_n$ denotes the involution 
\[
(z_0:z_1:\ldots:z_n)\dashrightarrow\left(\prod_{\stackrel{i=0}{i\not=0}}^nz_i:\prod_{\stackrel{i=0}{i\not=1}}^nz_i:\ldots:\prod_{\stackrel{i=0}{i\not=n}}^nz_i\right). 
\]
Hudson also deals with this group (\cite{Hudson}):
\begin{quotation}
"For a general space transformation, there is nothing to answer either to a plane characteristic or Noether theorem. There is however a group of transformations, called punctual because each is determined by a set of points, which are defined to satisfy an analogue of Noether theorem, and possess characteristics, and for which we can set up parallels to a good deal of the plane theory."
\end{quotation}

Note that the maps of $G_3(\mathbb{C})$ are in fact not so "punctual" (\cite[\S 8]{BlancHeden}). It follows from Noether theorem (\cite{AlberichCarraminana, Shafarevich}) that $G_2(\mathbb{C})$ coincides with $\mathrm{Bir}(\mathbb{P}^2_\mathbb{C})$; it is not the case in higher dimension where~$G_n(\mathbb{C})$ is a strict subgroup of $\mathrm{Bir}(\mathbb{P}^n_\mathbb{C})$ (\emph{see} \cite{Hudson, Pan}). However the following theorems show that $G_n(\mathbb{C})$ shares good properties with $G_2(\mathbb{C})=\mathrm{Bir}(\mathbb{P}^2_\mathbb{C})$.

\medskip

In \cite{CerveauDeserti} we proved that for any integer $n\geq 2$ the group $\mathrm{Bir}(\mathbb{P}^n_\Bbbk)$, where $\Bbbk$ denotes an algebraically closed field, is not linear; we obtain a similar statement for $G_n(\Bbbk)$, $n\geq 2$:

\begin{theoalph}\label{thm:subgroup}
{\sl If $\Bbbk$ is an algebraically closed field, there is no nontrivial finite-dimensional linear representation of $G_n(\Bbbk)$ over any field.}
\end{theoalph}

The group $G_n(\mathbb{C})$ contains some "big" subgroups:

\begin{proalph}\label{pro:subgroup}
{\sl  
\begin{itemize}
\item The group of polynomial automorphisms of $\mathbb{C}^n$ generated by the affine automorphisms and the Jonqui\`eres ones is a subgroup of $G_n(\mathbb{C})$.

\item If $\mathfrak{g}_0$, $\mathfrak{g}_1$, $\ldots$, $\mathfrak{g}_k$ are some generic automorphisms of~$\mathbb{P}^n_\mathbb{C}$, then $\langle \mathfrak{g}_0\sigma_n,\,\mathfrak{g}_1\sigma_n,\,\ldots,\,\mathfrak{g}_k\sigma_n\rangle\subset G_n(\mathbb{C})$ is a free subgroup of $G_n(\mathbb{C})$.
\end{itemize}}
\end{proalph}

\begin{rem}
For the meaning of "generic" see the proof of Proposition \ref{Pro:freegroup}.
\end{rem}

In \cite{CerveauDeserti} we establish that $G_2(\mathbb{C})=\mathrm{Bir}(\mathbb{P}^2_\mathbb{C})$ is perfect, \emph{i.e.} $[G_2(\mathbb{C}),G_2(\mathbb{C})]=G_2(\mathbb{C})$; the same holds for any $n$:

\begin{theoalph}\label{thm:ptesgpes1}
{\sl If $\Bbbk$ is an algebraically closed field, $G_n(\Bbbk)$ is perfect.}
\end{theoalph}

In \cite{Deserti:autbir} we determine the automorphisms group of $G_2(\mathbb{C})=\mathrm{Bir}(\mathbb{P}^2_\mathbb{C})$; in higher dimensions we have a similar description. Before giving a precise result, let us introduce some notation: the group of the field automorphisms acts on $\mathrm{Bir}(\mathbb{P}^n_\mathbb{C})$: if $f$ is an element of~$\mathrm{Bir}(\mathbb{P}^n_\mathbb{C})$, and $\kappa$ is a field automorphism we denote by~${}^{\kappa}\!\, f$ the element obtained by letting~$\kappa$ acting on $f$. 

\begin{theoalph}\label{thm:ptesgpes2}
{\sl Let $\varphi$ be an automorphism of $\mathrm{Bir}(\mathbb{P}^n_\mathbb{C})$. There exist $\kappa$ an automorphism of the field $\mathbb{C}$, and $\psi$ a birational map of $\mathbb{P}^n_\mathbb{C}$ such that 
\[
\varphi(f)={}^{\kappa}\!\,(\psi f\psi^{-1}) \qquad \forall\,f\in G_n(\mathbb{C}).
\]}
\end{theoalph}

The question "is the Cremona group simple ?" is a very old one; Cantat and Lamy recently gave a negative answer in dimension $2$ (\emph{see} \cite{CantatLamy}). One can consider the same question when $G_2(\Bbbk)$ is equipped with the Zariski topology ($\Bbbk$ denotes here an algebraically closed field) ; Blanc looked at it, and obtained a positive answer (\cite{Blanc}). What about $G_n(\Bbbk)$~?

\begin{proalph}\label{thm:ptesgpes3}
{\sl If $\Bbbk$ is an algebraically closed field, the group $G_n(\Bbbk)$, equipped with the Zariski topology, is simple.}
\end{proalph}

\subsubsection*{Organisation of the article} 
We first recall a result of Pan about the set of group ge\-nerators of~$\mathrm{Bir}(\mathbb{P}^n_\mathbb{C})$, $n\geq 3$ (\emph{see} \S \ref{sec:pan}); we then note that as soon as $n\geq 3$, there are birational maps of degree $n=\deg\sigma_n$ that do not belong to $G_n(\mathbb{C})$. In \S \ref{sec:nonlin} we prove Theorem~\ref{thm:subgroup}, and in \S \ref{sec:subgroup} Proposition \ref{pro:subgroup}. Let us remark that the fact that the group of tame automorphisms is contained in $G_n(\mathbb{C})$ implies that $G_n(\mathbb{C})$ contains maps of any degree, it was not obvious \emph{a priori}. In \S \ref{sec:algebraic} we study the normal subgroup in $G_n(\mathbb{C})$ generated by~$\sigma_n$ (resp. by an automorphism of $\mathbb{P}^n_\mathbb{C}$); it allows us to establish Theorem \ref{thm:ptesgpes1}. We finish \S\ref{sec:algebraic} with the proofs of Theorem \ref{thm:ptesgpes2}, and Proposition \ref{thm:ptesgpes3}.

\subsubsection*{Acknowledgments} 
I would like to thank D. Cerveau for his helpful and continuous listening. Thanks to the referee that helps me to improve the exposition. Thanks to I. Dolgachev for pointing out me that Coble introduced the group $G_n(\mathbb{C})$ in \cite{Coble}, and to J. Blanc, J. Diller,  F. Han, M. Jonsson, J.-L. Lin for their remarks and comments.

\section{About the set of group generators of $\mathrm{Bir}(\mathbb{P}^n_\mathbb{C})$, $n\geq 3$}\label{sec:pan}

\subsection{Some definitions}

A \textbf{\textit{polynomial automorphism}} $\varphi$ of $\mathbb{C}^n$ is a map $\mathbb{C}^n\to\mathbb{C}^n$ of the type 
\[
(z_0,z_1,\ldots,z_{n-1})\mapsto\big(\varphi_0(z_0,z_1,\ldots,z_{n-1}),\varphi_1(z_0,z_1,\ldots,z_{n-1}),\ldots,\varphi_{n-1}(z_0,z_1,\ldots,z_{n-1})\big),
\]
with $\varphi_i\in\mathbb{C}[z_0,z_1,\ldots,z_{n-1}]$, that is bijective; we denote $\varphi$ by $\varphi=(\varphi_0,\varphi_1,\ldots,\varphi_{n-1})$. A \textbf{\textit{rational self-map}} $\phi\colon\mathbb{P}^n_\mathbb{C}\dashrightarrow\mathbb{P}^n_\mathbb{C}$ is given by
\[
(z_0:z_1:\ldots:z_n)\dashrightarrow\big(\phi_0(z_0,z_1,\ldots,z_n):\phi_1(z_0,z_1,\ldots,z_n):\ldots:\phi_n(z_0,z_1,\ldots,z_n)\big)
\]
where the $\phi_i$ are homogeneous polynomials of the same positive degree, and without common factor of positive degree. Let us denote by $\mathbb{C}[z_0,z_1,\ldots,z_n]_d$ the set of homogeneous polynomials in $z_0$, $z_1$, $\ldots$, $z_n$ of degree $d$. The \textbf{\textit{degree}} of $\phi$ is by definition the degree of the~$\phi_i$. A \textbf{\textit{birational self-map}} of $\mathbb{P}^n_\mathbb{C}$ is a rational self-map that admits a rational inverse. The set of polynomial automorphisms of $\mathbb{C}^n$ (resp. birational self-maps of $\mathbb{P}^n_\mathbb{C}$) form a group denoted $\mathrm{Aut}(\mathbb{C}^n)$ (resp. $\mathrm{Bir}(\mathbb{P}^n_\mathbb{C})$).

\subsection{A result of Pan}\label{subsec:hudsonpan}

Let us recall a construction of Pan (\cite{Pan}) which, given a birational self-map of $\mathbb{P}^n_\mathbb{C}$, allows one to construct a birational self-map of $\mathbb{P}^{n+1}_\mathbb{C}$. Let $P\in\mathbb{C}[z_0,z_1,\ldots,z_n]_d$, $Q\in\mathbb{C}[z_0,z_1,\ldots,z_n]_\ell$, and let $R_0$, $R_1$, $\ldots$, $R_{n-1}\in\mathbb{C}[z_0,z_1,\ldots,z_{n-1}]_{d-\ell}$ be some homogeneous polynomials.
Denote by $\Psi_{P,Q,R}\colon\mathbb{P}^n_\mathbb{C}\dashrightarrow\mathbb{P}^n_\mathbb{C}$ and $\widetilde{\Psi}\colon\mathbb{P}^{n-1}_\mathbb{C}\dashrightarrow\mathbb{P}^{n-1}_\mathbb{C}$ the rational maps defined by 
\[
\Psi_{P,Q,R}=\big(QR_0:QR_1:\ldots:QR_{n-1}:P\big)\qquad \&\qquad \widetilde{\Psi}_R=\big(R_0:R_1:\ldots:R_{n-1}\big).
\]

\begin{lem}[\cite{Pan}]\label{lem:pan0}
{\sl Let $d$, $\ell$ be some integers such that $d\geq \ell+1\geq 2$. Take $Q$ in $\mathbb{C}[z_0,z_1,\ldots,z_n]_\ell$, and $P$ in $\mathbb{C}[z_0,z_1,\ldots,z_n]_d$ without common factors. Let $R_1$, $\ldots$, $R_n$ be some elements of~$\mathbb{C}[z_0,z_1,\ldots,z_{n-1}]_{d-\ell}$. Assume that 
\[
P=z_nP_{d-1}+P_d \qquad\qquad Q=z_nQ_{\ell-1}+Q_\ell
\]
with $P_{d-1}$, $P_d$, $Q_{\ell-1}$, $Q_\ell\in\mathbb{C}[z_0,z_1,\ldots,z_{n-1}]$ of degree $d-1$, resp. $d$, resp. $\ell-1$, resp. $\ell$ and such that $(P_{d-1},Q_{\ell-1})\not=(0,0)$. 

The map $\Psi_{P,Q,R}$ is birational if and only if $\widetilde{\Psi}_R$ is birational.}
\end{lem}

Let us give the motivation of this construction:

\begin{thm}[\cite{Hudson, Pan}]\label{thm:hudsonpan}
{\sl Any set of group generators of $\mathrm{Bir}(\mathbb{P}^n_\mathbb{C})$, $n\geq 3$, contains uncountably many non-linear maps.}
\end{thm}

We will give an idea of the proof of this statement.

\begin{lem}[\cite{Pan}]\label{lem:pan}
{\sl Let $n\geq 3$. Let $\mathcal{S}$ be an hypersurface of $\mathbb{P}^n_\mathbb{C}$ of degree $\ell\geq 1$ having a point~$p$ of multiplicity $\geq \ell-1$. 

Then there exists a birational self-map of $\mathbb{P}^n_\mathbb{C}$ of degree $d\geq \ell+1$ that blows down $\mathcal{S}$ onto a point.}
\end{lem}

\begin{proof}
One can assume without loss of generality that $p=(0:0:\ldots:0:1)$. Denote by $q'=0$ the equation of $\mathcal{S}$, and take a generic plane passing through $p$ given by the equation $h=0$. Finally choose $P=z_nP_{d-1}+P_d$ such that
\smallskip
\begin{itemize}
\item[$\bullet$] $P_{d-1}\not=0$;
\smallskip
\item[$\bullet$] $\textrm{pgcd}\,(P,hq')=1$.
\end{itemize}
\smallskip
Now set $Q=h^{d-\ell-1}q'$, $R_i=z_i$, and conclude with Lemma \ref{lem:pan0}.
\end{proof}

\begin{proof}[Proof of Theorem \ref{thm:hudsonpan}]
Let us consider the family of hypersurfaces given by $q(z_1,z_2,z_3)=0$ where $q=0$ defines a smooth curve $\mathcal{C}_q$ of degree $\ell$ on $\{z_0=z_4=z_5=\ldots=z_n=0\}$. Let us note that $q=0$ is birationally equivalent to $\mathbb{P}^{n-2}_\mathbb{C}\times\mathcal{C}_q$. Furthermore $q=0$ and $q'=0$ are birationally equivalent if and only if $\mathcal{C}_q$ and $\mathcal{C}_{q'}$ are isomorphic. Note that for $\ell=2$ the set of isomorphism classes of smooth cubics is a $1$-parameter family, and that according to Lemma \ref{lem:pan} for any $\mathcal{C}_q$ there exists a birational self-map of $\mathbb{P}^n_\mathbb{C}$ that blows down $\mathcal{C}_q$ onto a point. Hence any set of group generators of $\mathrm{Bir}(\mathbb{P}^n_\mathbb{C})$, $n\geq 3$, has to contain uncountably many non-linear maps.
\end{proof}

One can take $d=\ell+1$ in Lemma \ref{lem:pan}. In particular

\begin{cor}
{\sl As soon as $n\geq 3$, there are birational maps of degree $n=\deg\sigma_n$ that do not belong to $G_n(\mathbb{C})$.}
\end{cor}

\begin{rem}
The maps $\Psi_{P,Q,R}$ that are birational form a subgroup of $\mathrm{Bir}(\mathbb{P}^n_\mathbb{C})$ denoted by~$\mathrm{J}_0(1;\mathbb{P}^n_\mathbb{C})$, and stu\-died in \cite{PanSimis} : in particular $\mathrm{J}_0(1;\mathbb{P}^3_\mathbb{C})$ inherits the property of Theorem~\ref{thm:hudsonpan}.
\end{rem}

\subsection{A first remark}

Let $\phi$ be a birational map of $\mathbb{P}^3_\mathbb{C}$. A \textbf{\textit{regular resolution}} of $\phi$ is a morphism $\pi\colon Z\to \mathbb{P}^3_\mathbb{C}$ which is a sequence of blow-ups 
\[
\pi=\pi_1\circ\ldots\circ\pi_r
\]
along smooth irreducible centers, such that 
\begin{itemize}
\item $\phi\circ\pi\colon Z\to\mathbb{P}^3_\mathbb{C}$ is a birational morphism,
\item and each center $B_i$ of the blow-up $\pi_i\colon Z_i\to Z_{i-1}$ is contained in the base locus of the induced map $Z_{i-1}\dashrightarrow\mathbb{P}^3_\mathbb{C}$. 
\end{itemize}
It follows from Hironaka that such a resolution always exists. If $B$ is a smooth irreducible center of a blow-up in a smooth projective complex variety of dimension $3$, then $B$ is either a point, or a smooth curve. We define the genus of $B$ as follows: it is $0$ if $B$ is a point, the usual genus otherwise. Frumkin defines the \textbf{\textit{genus}} of $\phi$ to be the maximum of the genera of the centers of the blow-ups in the resolution of $\phi$ (\emph{see} \cite{Frumkin}), and shows that this definition does not depend on the choice of the regular resolution. In \cite{Lamy} an other definition of the genus of a birational map is given. Let us recall that if $E$ is an irreducible divisor contracted by a birational map between smooth projective complex varieties of dimension $3$, then $E$ is birational to $\mathbb{P}^1_\mathbb{C}\times\mathcal{C}$, where $\mathcal{C}$ denotes a smooth curve (\cite{Lamy}). The genus of a birational map $\phi$ of $\mathbb{P}^3_\mathbb{C}$ is the maximum of the genera of the irreducible divisors in $\mathbb{P}^3_\mathbb{C}$ contracted by $\phi$. Lamy proves that these two definitions of genus agree (\cite{Lamy}).

Let $\phi$ be in $\mathrm{Bir}(\mathbb{P}^3_\mathbb{C})$, and let $\mathcal{H}$ be an irreducible hypersurface of $\mathbb{P}^3_\mathbb{C}$. We say that $\mathcal{H}$ is \textbf{\textit{$\phi$-exceptional}} if $\phi$ is not injective on any open subset of $\mathcal{H}$ (or equivalently if there is an open subset of $\mathcal{H}$ which is mapped into a subset of codimension $\geq 2$ by $\phi$). Let $\phi_1$, $\ldots$, $\phi_k$ be in $\mathrm{Bir}(\mathbb{P}^3_\mathbb{C})$, and let $\phi=\phi_k\circ\ldots\circ\phi_1$. Let $\mathcal{H}$ be an irreducible hypersurface of $\mathbb{P}^3_\mathbb{C}$. If $\mathcal{H}$ is $\phi$-exceptional, then there exists $1\leq i\leq k$ and a $\phi_i$-exceptional hypersurface $\mathcal{H}_i$ such that
\begin{itemize}
\item $\phi_{i-1}\circ\ldots\circ\phi_1$ realizes a birational isomorphism from $\mathcal{H}$ to $\mathcal{H}_i$;

\item $\phi_i$ contracts $\mathcal{H}_i$.
\end{itemize}

In particular one has the following statement.

\begin{pro}
{\sl The group $G_3(\mathbb{C})$ is contained in the subgroup of birational self-maps of $\mathbb{P}^3_\mathbb{C}$ of genus $0$.}
\end{pro}

\section{Non-linearity of $G_n(\mathbb{C})$}\label{sec:nonlin}

If $V$ is a finite dimensional vector space over~$\mathbb{C}$ there is no faithful linear representation $\mathrm{Bir}(\mathbb{P}^n_\mathbb{C})\to\mathrm{GL}(V)$ (\emph{see} \cite[Proposition 5.1]{CerveauDeserti}). The proof of this statement is based on the following Lemma due to Birkhoff (\cite[Lemma 1]{Birkhoff}): if $\mathfrak{a}$, $\mathfrak{b}$ and $\mathfrak{c}$ are three elements of $\mathrm{GL}(n;\mathbb{C})$ such that 
\[
[\mathfrak{a},\mathfrak{b}]=\mathfrak{c},\quad [\mathfrak{a},\mathfrak{c}]=[\mathfrak{b},\mathfrak{c}]=\mathrm{id},\quad \mathfrak{c}^p=\mathrm{id} \text{ for some $p$ prime}
\]
then $p\leq n$. Assume that there exists an injective homomorphism $\rho$ from $\mathrm{Bir}(\mathbb{P}^2_\mathbb{C})$ to $\mathrm{GL}(n;\mathbb{C})$. For any $p>n$ prime consider in the affine chart $z_2=1$ the maps 
\[
(\exp(2\mathbf{i}\pi/p)z_0,z_1),\qquad (z_0,z_0z_1),\qquad(z_0,\exp(-2\mathbf{i}\pi/p)z_1).
\]
The image by $\rho$ of these maps satisfy Birkhoff Lemma so $p\leq n$: contradiction. In any dimension we have the same property: $G_n(\mathbb{C})$ is not linear, \emph{i.e.}  if $V$ is a finite dimensional vector space over $\mathbb{C}$ there is no faithful linear representation $G_n(\mathbb{C})\to\mathrm{GL}(V)$. Actually $G_n(\mathbb{C})$ satisfies a more precise property due to Cornulier in dimension $2$ (\emph{see}~\cite{Cornulier}):

\begin{pro}\label{pro:nonlin2}
{\sl The group $G_n(\mathbb{C})$ has no non-trivial finite dimensional representation.}
\end{pro}

\begin{lem}\label{lem:pratique}
{\sl The map $\varsigma=\big(z_0z_{n-1}:z_1z_{n-1}:\ldots:z_{n-2}z_{n-1}: z_{n-1}z_n:z_n^2\big)$ belongs to $G_n(\mathbb{C})$.}
\end{lem}

\begin{proof}
We have $\varsigma=\mathfrak{a}_1\sigma_n\mathfrak{a}_2\sigma_n\mathfrak{a}_3$ where
\[
\mathfrak{a}_1=\big(z_2-z_1:z_3-z_1:\ldots:z_n-z_1:z_1:z_1-z_0\big),
\]
\[
\mathfrak{a}_2=\big(z_{n-1}+z_n:z_n:z_0:z_1:\ldots:z_{n-2}\big),
\]
\[
\mathfrak{a}_3=\big(z_0+z_n:z_1+z_n:\ldots:z_{n-2}+z_n:z_{n-1}-z_n:z_n\big).
\]
\end{proof}

\begin{proof}[Proof of Proposition \ref{pro:nonlin2}]
We adapt the proof of \cite{Cornulier}.

Let us now work in the affine chart $z_n=1$. By Lemma \ref{lem:pratique} in $G_n(\mathbb{C})$ there is a natural copy of $H=(\mathbb{C}^*)^n\rtimes\mathbb{Z}$; indeed $\langle\varsigma=\big(z_0z_{n-1},z_1z_{n-1},\ldots,z_{n-2}z_{n-1}, z_{n-1}\big)\rangle\simeq\mathbb{Z}$ acts on $\big\{(\alpha_0z_0,\alpha_1z_1,\ldots,\alpha_{n-1}z_{n-1})\,\vert\,\alpha_i\in\mathbb{C}^*\big\}\simeq(\mathbb{C}^*)^n$ and $H$ is the group of maps 
\[
\big\{(\alpha_0z_0z_{n-1}^k,\alpha_1z_1z_{n-1}^k,\ldots,\alpha_{n-2}z_{n-2}z_{n-1}^k,\alpha_{n-1}z_{n-1})\,\vert\,\alpha_i\in\mathbb{C}^*,\,k\in\mathbb{Z}\big\}.
\]

Consider any linear representation $\rho\colon H\to\mathrm{GL}(k;\mathbb{C})$. If $p$ is prime, and if $\xi_p$ is a primitive $p$-root of unity, set  
\[
\mathfrak{g}_p=(\xi_pz_0,\xi_pz_1,\ldots,\xi_pz_{n-1}), \quad \mathfrak{h}_p=(\xi_p z_0,\xi_p z_1,\ldots,\xi_p z_{n-2},z_{n-1}).
\]
Then $\mathfrak{h}_p=[\varsigma,\mathfrak{g}_p]$ commutes with both $\phi$ and $\mathfrak{g}_p$. By \cite[Lemma 1]{Birkhoff} if $\rho(\mathfrak{g}_p)\not=1$, then $k\geq p$.

Picking $p$ to be greater than $k$, this shows that if we have an arbitrary representation $f\colon G_n(\mathbb{C})\to\mathrm{GL}(k;\mathbb{C})$, the restriction $f_{\vert\mathrm{PGL}(n+1;\mathbb{C})}$ is not faithful. Since $\mathrm{PGL}(n+1;\mathbb{C})$ is simple, this implies that $f$ is trivial on $\mathrm{PGL}(n+1;\mathbb{C})$. We conclude by using the fact that the two involutions $-\mathrm{id}$ and~$\sigma_n$ are conjugate via the map $\psi$ given by
\begin{small}
\[
\left(\frac{z_0+1}{z_0-1},\frac{z_1+1}{z_1-1},\ldots,\frac{z_{n-1}+1}{z_{n-1}-1}\right)
\]
\end{small}
and $\psi=\mathfrak{a}_1\sigma_n\mathfrak{a}_2$ where $\mathfrak{a}_1$ and $\mathfrak{a}_2$ denote the two following automorphisms of $\mathbb{P}^n_\mathbb{C}$ 
\[
\mathfrak{a}_1=\big(z_0+1,z_1+1,\ldots,z_{n-1}+1\big),
\]
\[
\mathfrak{a}_2=\Big(\frac{z_0-1}{2},\frac{z_1-1}{2},\ldots,\frac{z_{n-1}-1}{2}\Big).
\]
\end{proof}

\begin{rem}
Proposition \ref{pro:nonlin2} is also true for $G_n(\Bbbk)$ where $\Bbbk$ is an algebraically closed field.
\end{rem}

\section{Subgroups of $G_n(\mathbb{C})$}\label{sec:subgroup}

\subsection{The tame automorphisms}

The automorphisms of $\mathbb{C}^n$ written in the form $(\phi_0,\phi_1,\ldots,\phi_{n-1})$ where 
\[
\phi_i=\phi_i(z_i,z_{i+1},\ldots,z_{n-1})
\]
depends only on $z_i$, $z_{i+1}$, $\ldots$, $z_{n-1}$ form the \textbf{\textit{Jonqui\`eres subgroup}} $\mathrm{J}_n\subset\mathrm{Aut}(\mathbb{C}^n)$. A polynomial automorphism  $(\phi_0,\phi_1,\ldots,\phi_{n-1})$ where all the~$\phi_i$ are linear is \textbf{\textit{an affine transformation}}. Denote by $\mathrm{Aff}_n$ the \textbf{\textit{group of affine transformations}}; $\mathrm{Aff}_n$ is the semi-direct product of $\mathrm{GL}(n;\mathbb{C})$ with the commutative unipotent subgroup of translations. We have the following inclusions
\[
\mathrm{GL}(n;\mathbb{C})\subset\mathrm{Aff}_n\subset\mathrm{Aut}(\mathbb{C}^n).
\] 
The subgroup $\mathrm{Tame}_n\subset\mathrm{Aut}(\mathbb{C}^n)$ generated by $\mathrm{J}_n$ and $\mathrm{Aff}_n$ is called the \textbf{\textit{group of tame automorphisms}}. For $n=2$ one has $\mathrm{Tame}_2=\mathrm{Aut}(\mathbb{C}^2)$, this follows from the fact that $\mathrm{Aut}(\mathbb{C}^2)=\mathrm{J}_2\ast_{\mathrm{J}_2\cap\mathrm{Aff}_2}\mathrm{Aff}_2$ (\emph{see} \cite{Jung}). The group $\mathrm{Tame}_3$ does not coincide with $\mathrm{Aut}(\mathbb{C}^3)$: the Nagata automorphism is not tame (\cite{ShestakovUmirbaev}). Derksen gives a set of generators of $\mathrm{Tame}_n$ (\emph{see} \cite{VandenEssen} for a proof):

\begin{thm}\label{thm:derksen}
{\sl Let $n\geq 3$ be a natural integer. The group $\mathrm{Tame}_n$ is generated by $\mathrm{Aff}_n$, and the Jonqui\`eres map $\big(z_0+z_1^2,z_1,z_2,\ldots,z_{n-1}\big)$.}
\end{thm}

\begin{pro}
{\sl The group $G_n(\mathbb{C})$ contains the group of tame polynomial automorphisms of~$\mathbb{C}^n$.}
\end{pro}

\begin{proof}
The inclusion $\mathrm{Aff}_n\subset\mathrm{Aut}(\mathbb{P}^n_\mathbb{C})$ is obvious; according to Theorem \ref{thm:derksen} we thus just have to prove that $\big(z_0+z_1^2,z_1,z_2,\ldots,z_{n-1}\big)$ belongs to $G_n(\mathbb{C})$. But 
\[
\big(z_0z_n+z_1^2:z_1z_n:z_2z_n:\ldots:z_{n-1}z_n:z_n^2\big)=\mathfrak{g}_1\sigma_n \mathfrak{g}_2\sigma_n \mathfrak{g}_3\sigma_n \mathfrak{g}_2\sigma_n\mathfrak{g}_4
\] 
where
\begin{align*}
& \mathfrak{g}_1=\big(z_2-z_1+z_0:2z_1-z_0:z_3:z_4:\ldots:z_n:z_1-z_0\big), \\
& \mathfrak{g}_2=\big(z_0+z_2:z_0:z_1:z_3:z_4:\ldots:z_n\big), \\
& \mathfrak{g}_3=\big(-z_1:z_0+z_2-3z_1:z_0:z_3:z_4:\ldots:z_n\big), \\
& \mathfrak{g}_4=\big(z_1-z_n:-2z_n-z_0:2z_n-z_1:-z_2:-z_3:\ldots:-z_{n-1}\big). 
\end{align*}
\end{proof}

\subsection{Free groups and $G_n(\mathbb{C})$}

Following the idea of \cite[Proposition 5.7]{CerveauDeserti} we prove that:

\begin{pro}\label{Pro:freegroup}
{\sl Let $\mathfrak{g}_0$, $\mathfrak{g}_1$, $\ldots$, $\mathfrak{g}_k$ be some generic elements of $\mathrm{Aut}(\mathbb{P}^n_\mathbb{C})$. The group generated by $\mathfrak{g}_0$, $\mathfrak{g}_1$, $\ldots$, $\mathfrak{g}_k$, and $\sigma_n$ is the free product
\[
\overbrace{\mathbb{Z}\ast\ldots\ast\mathbb{Z}}^{k+1}\,\ast\,(\mathbb{Z}/2\mathbb{Z}),
\]
the $\mathfrak{g}_i$'s and $\sigma_n$ being the generators for the factors of this free product.

In particular the subgroup $\langle \mathfrak{g}_0\sigma_n,\,\mathfrak{g}_1\sigma_n,\,\ldots,\,\mathfrak{g}_k\sigma_n\rangle$ of $G_n(\mathbb{C})$ is a free group.
}
\end{pro}

\begin{rem}
The meaning of "generic" is explained in the proof below.
\end{rem}

\begin{proof}
Let us show the statement for $k=0$ (in the general case it is sufficient to replace the free product $\mathbb{Z}\ast\mathbb{Z}/2\mathbb{Z}$ by $\mathbb{Z}\ast\mathbb{Z}\ast\ldots\ast\mathbb{Z}\ast\mathbb{Z}/2\mathbb{Z}$).

If $\langle \mathfrak{g},\sigma_n\rangle$ is not isomorphic to $\mathbb{Z}\ast\mathbb{Z}/2\mathbb{Z}$, then there exists a word $M_\mathfrak{g}$ in $\mathbb{Z}\ast\mathbb{Z}/2\mathbb{Z}$ such that $M_\mathfrak{g}(\mathfrak{g},\sigma_n)=\mathrm{id}$. Note that the set of words $M_\mathfrak{g}$ is countable, and that for a given word $M$ the set
\[
R_M=\big\{\mathfrak{g}\,\big\vert\, M(\mathfrak{g},\sigma_n)=\mathrm{id}\big\}
\] 
is algebraic in $\mathrm{Aut}(\mathbb{P}^n_\mathbb{C})$. Consider an automorphism $\mathfrak{g}$ written in the following form 
\[
\big(\alpha z_0+\beta z_1:\gamma z_0+\delta z_1:z_2:z_3:\ldots:z_n\big)
\]
where $\left[\begin{array}{cc}
\alpha & \beta \\
\gamma & \delta
\end{array}\right]\in\mathrm{PGL}(2;\mathbb{C})$. Since the pencil $z_0=tz_1$ is invariant by both $\sigma_n$ and $\mathfrak{g}$, one inherits a linear representation
\[
\langle \mathfrak{g},\,\sigma_n\rangle\to\mathrm{PGL}(2;\mathbb{C})
\]
defined by 
\[
\mathfrak{g}\colon t\mapsto \frac{\alpha t+\beta}{\gamma t+\delta},\qquad \sigma_n\colon t\mapsto \frac{1}{t}.
\] 
But the group generated by $\left[\begin{array}{cc} \alpha & \beta \\ \gamma & \delta \end{array}\right]$ and $\left[\begin{array}{cc} 0 & 1 \\ 1 & 0 \end{array}\right]$ is generically isomorphic to $\mathbb{Z}\ast\mathbb{Z}/2\mathbb{Z}$ (\emph{see}~\cite{delaHarpe}). Hence the complements $R_M^C$ are dense open subsets, and their intersection is dense by Baire property.
\end{proof}

\section{Some algebraic properties of $G_n(\mathbb{C})$}\label{sec:algebraic}

\subsection{The group $G_n(\mathbb{C})$ is perfect}

If $\mathrm{G}$ is a group, and if $g$ is an element of $\mathrm{G}$, we denote by 
\[
\mathrm{N}(g;\mathrm{G})=\langle fgf^{-1}\,\vert\,f\in\mathrm{G}\rangle.
\]
the normal subgroup generated by $g$ in $\mathrm{G}$.

\begin{pro}\label{pro:ssgpenormeng}
{\sl The following assertions hold:
\smallskip
\begin{enumerate}
\item $\mathrm{N}(\mathfrak{g};\mathrm{PGL}(n+1;\mathbb{C}))=\mathrm{PGL}(n+1;\mathbb{C})$ for any $\mathfrak{g}\in\mathrm{PGL}(n+1;\mathbb{C})\smallsetminus\{\mathrm{id}\}$;
\smallskip
\item $\mathrm{N}(\sigma_n;G_n(\mathbb{C}))=G_n(\mathbb{C})$;
\smallskip
\item $\mathrm{N}(\mathfrak{g};G_n(\mathbb{C}))=G_n(\mathbb{C})$ for any $\mathfrak{g}\in\mathrm{PGL}(n+1;\mathbb{C})\smallsetminus\{\mathrm{id}\}$.
\end{enumerate}}
\end{pro}

\begin{proof}
Let us work in the affine chart $z_n=1$.
\begin{enumerate}
\item Since $\mathrm{PGL}(n+1;\mathbb{C})$ is simple one has the first assertion.

\medskip

\item Let $\phi$ be in $G_n(\mathbb{C})$; there exist $\mathfrak{g}_0$, $\mathfrak{g}_1$, $\ldots$, $\mathfrak{g}_k$ in $\mathrm{Aut}(\mathbb{P}^n_\mathbb{C})$ such that
\[
\phi=(\mathfrak{g}_0)\,\sigma_n\, \mathfrak{g}_1\,\sigma_n\,\ldots \sigma_n\,\mathfrak{g}_k\,(\sigma_n).
\]
As $\mathrm{PGL}(n+1;\mathbb{C})$ is simple 
\[
\mathrm{N}(-\mathrm{id};\mathrm{PGL}(n+1;\mathbb{C}))=\mathrm{PGL}(n+1;\mathbb{C}), 
\]
and for any $0\leq i\leq k$ there exist $\mathfrak{f}_{i,0}$, $\mathfrak{f}_{i,1}$, $\ldots$, $\mathfrak{f}_{i,\ell_i}$ in $\mathrm{PGL}(n+1;\mathbb{C})$ such that 
\[
\mathfrak{g}_i=\mathfrak{f}_{i,0}\big(-\mathrm{id}\big)\mathfrak{f}_{i,0}^{-1}\,\mathfrak{f}_{i,1}\big(-\mathrm{id}\big)\mathfrak{f}_{i,1}^{-1}\,\ldots\,\mathfrak{f}_{i,\ell_i}\big(-\mathrm{id}\big)\mathfrak{f}_{i,\ell_i}^{-1}.
\]
We conclude by using the fact that $-\mathrm{id}$ and $\sigma_n$ are conjugate via an element of $G_n(\mathbb{C})$ (\emph{see} the proof of Proposition \ref{pro:nonlin2}). 

\medskip

\item Fix $\mathfrak{g}$ in $\mathrm{PGL}(n+1;\mathbb{C})\smallsetminus\{\mathrm{id}\}$. Since $\mathrm{N}(\mathfrak{g};\mathrm{PGL}(n+1;\mathbb{C}))=\mathrm{PGL}(n+1;\mathbb{C})$, the involution~$-\mathrm{id}$ can be written as a composition of some conjugates of $\mathfrak{g}$. The maps $-\mathrm{id}$ and $\sigma_n$ being conjugate one has
\[
\sigma_n=(f_0 \mathfrak{g}f_0^{-1})\,(f_1 \mathfrak{g}f_1^{-1})\,\ldots\,(f_\ell \mathfrak{g}f_\ell^{-1})
\]
for some $f_i$ in $G_n(\mathbb{C})$. So $\mathrm{N}(\sigma_n;G_n(\mathbb{C}))\subset\mathrm{N}(\mathfrak{g};G_n(\mathbb{C}))$, and one concludes with the second assertion.
\end{enumerate}
\end{proof}

\begin{cor}\label{cor:prodconj}
{\sl The group $G_n(\mathbb{C})$ satisfies the following properties:
\smallskip
\begin{enumerate}
\item $G_n(\mathbb{C})$ is perfect, \emph{i.e.} $[G_n(\mathbb{C}),G_n(\mathbb{C})]=G_n(\mathbb{C})$;
\smallskip
\item for any $\phi$ in $G_n(\mathbb{C})$ there exist $\mathfrak{g}_0$, $\mathfrak{g}_1$, $\ldots$, $\mathfrak{g}_k$ automorphisms of $\mathbb{P}^n_\mathbb{C}$ such that 
\[
\phi=(\mathfrak{g}_0\sigma_n\mathfrak{g}_0^{-1})(\mathfrak{g}_1\sigma_n\mathfrak{g}_1^{-1})\ldots(\mathfrak{g}_k\sigma_n\mathfrak{g}_k^{-1})
\]
\end{enumerate}}
\end{cor}

\begin{proof}
\begin{enumerate}
\item The third assertion of Proposition \ref{pro:ssgpenormeng} implies that any element of~$G_n(\mathbb{C})$ can be written as a composition of some conjugates of 
\[
\mathfrak{t}=\big(z_0:z_1+z_n:z_2+z_n:\ldots:z_{n-1}+z_n:z_n\big). 
\]
As
\begin{small}
\[
\mathfrak{t}=\Big[\big(z_0:3z_1:3z_2:\ldots:3z_{n-1}:z_n\big),\left(2z_0:z_1+z_n:z_2+z_n:\ldots:z_{n-1}+z_n:2z_n\right)\Big],
\]
\end{small}
the group $G_n(\mathbb{C})$ is perfect.
\medskip

\item For any $\alpha_0$, $\alpha_1$, $\ldots$, $\alpha_{n}$ in $\mathbb{C}^*$ set $\mathfrak{d}(\alpha_0,\alpha_1,\ldots,\alpha_{n})=(\alpha_0z_0:\alpha_1z_1:\ldots:\alpha_{n}z_{n})$, and let us define~$\mathrm{H}$ as follows: 
\[
\mathrm{H}=\Big\{\mathfrak{g}_0\sigma_n\mathfrak{g}_0^{-1}\,\mathfrak{g}_1\sigma_n\mathfrak{g}_1^{-1}\,\ldots\,\mathfrak{g}_\ell\sigma_n\mathfrak{g}_\ell^{-1}\,\vert\,\mathfrak{g}_i\in\mathrm{PGL}(n+1;\mathbb{C}),\,\ell\in\mathbb{N}\Big\}.
\]
The second assertion of the Corollary is then equivalent to $\mathrm{H}=G_n(\mathbb{C})$. Let us remark that $\mathrm{H}$ is a group that contains $\sigma_n$, and that $\mathrm{PGL}(n+1;\mathbb{C})$ acts by conjugacy on it. One can check that
\begin{equation}\label{eq:simpl}
\mathfrak{d}_\alpha\,\sigma_n\,\mathfrak{d}_\alpha^{-1}=\mathfrak{d}_\alpha^2\,\sigma_n=\sigma_n\,\mathfrak{d}_\alpha^{-2}.
\end{equation}
Hence for each $\mathfrak{g}$ in $\mathrm{PGL}(n+1;\mathbb{C})$ we have $\mathfrak{g}\mathfrak{d}_\alpha\,\sigma_n\,\mathfrak{d}_\alpha^{-1}\mathfrak{g}^{-1}=(\mathfrak{g}\mathfrak{d}_\alpha^2\mathfrak{g}^{-1})(\mathfrak{g}\sigma_n\mathfrak{g}^{-1})$, so $\mathfrak{g}\mathfrak{d}_\alpha^2\mathfrak{g}^{-1}$ belongs to $\mathrm{H}$. Since any automorphism of $\mathbb{P}^n_\mathbb{C}$ can be written as a product of diagonalizable matrices, $\mathrm{PGL}(n+1;\mathbb{C})\subset\mathrm{H}$.
\end{enumerate}
\end{proof}

\subsection{On the restriction of automorphisms of the group birational maps to $G_n(\mathbb{C})$} 

If $M$ is a projective variety defined over a field $\Bbbk\subset\mathbb{C}$ the group $\mathrm{Aut}_\Bbbk(\mathbb{C})$ of automorphisms of the field extension $\mathbb{C}/\Bbbk$ acts on $M(\mathbb{C})$, and on both $\mathrm{Aut}(M)$ and $\mathrm{Bir}(M)$ as follows
\begin{equation}\label{eq:fieldaut}
{}^{\kappa}\!\,\psi(p)=(\kappa\psi\kappa^{-1})(p)
\end{equation}
for any $\kappa$ in $\mathrm{Aut}_\Bbbk(\mathbb{C})$, any $\psi$ in $\mathrm{Bir}(M)$, and any point $p$ in $M(\mathbb{C})$ for with both sides of $(\ref{eq:fieldaut})$ are well defined. Hence $\mathrm{Aut}_\Bbbk(\mathbb{C})$ acts by automorphisms on $\mathrm{Bir}(M)$. If $\kappa\colon\mathbb{C}\to\mathbb{C}$ is a morphism field, this contruction gives an injective morphism
\[
\mathrm{Aut}(\mathbb{P}^n_\mathbb{C})\to\mathrm{Aut}(\mathbb{P}^n_\mathbb{C})\qquad \mathfrak{g}\mapsto \mathfrak{g}^\vee.
\]
Indeed, write $\mathbb{C}$ as the algebraic closure of a purely transcendental extension $\mathbb{Q}(x_i,i\in I)$ of $\mathbb{Q}$; if $f\colon I\to I$ is an injective map, then there exists a field morphism 
\[
\kappa\colon\mathbb{C}\to\mathbb{C}\qquad x_i\mapsto x_{f(i)}.
\]
Note that such a morphism is surjective if and only if $f$ is onto.

\medskip

In $2006$, using the structure of amalgamated product of $\mathrm{Aut}(\mathbb{C}^2)$, the automorphisms of this group have been described: 

\begin{thm}[\cite{Deserti:autaut}]\label{thm:autaut}
{\sl Let $\varphi$ be an automorphism of $\mathrm{Aut}(\mathbb{C}^2)$. There exist a polynomial automorphism $\psi$ of $\mathbb{C}^2$, and a field automorphism $\kappa$ such that
\[
\varphi(f)={}^{\kappa}\!\,(\psi f\psi^{-1})\qquad \forall\,f\in\mathrm{Aut}(\mathbb{C}^2).
\]}
\end{thm}

Then, in $2011$, Kraft and Stampfli show that every automorphism of $\mathrm{Aut}(\mathbb{C}^n)$ is inner up to field automorphisms when restricted to the group $\mathrm{Tame}_n$:

\begin{thm}[\cite{KraftStampfli}]
{\sl Let $\varphi$ be an automorphism of $\mathrm{Aut}(\mathbb{C}^n)$. There exist a polynomial automorphism $\psi$ of $\mathbb{C}^n$, and a field automorphism $\kappa$ such that
\[
\varphi(f)={}^{\kappa}\!\,(\psi f\psi^{-1})\qquad \forall\,f\in\mathrm{Tame}_n.
\]}
\end{thm}

Even if $\mathrm{Bir}(\mathbb{P}^2_\mathbb{C})$ hasn't the same structure as $\mathrm{Aut}(\mathbb{C}^2)$ (\emph{see} Appendix of \cite{CantatLamy}) the automorphisms group of $\mathrm{Bir}(\mathbb{P}^2_\mathbb{C})$ can be described, and a similar result as Theorem \ref{thm:autaut} is obtained (\cite{Deserti:autbir}). There is no such result in higher dimension; nevertheless in \cite{Cantat} Cantat classifies all (abstract) homomorphisms from $\mathrm{PGL}(k+1;\mathbb{C})$ to the group $\mathrm{Bir}(M)$ of birational maps of a complex projective variety $M$, provided $k\geq\dim_\mathbb{C}M$. Before recalling his statement let us introduce some notation. Given $\mathfrak{g}$ in $\mathrm{Aut}(\mathbb{P}^n_\mathbb{C})=\mathrm{PGL}(n+1;\mathbb{C})$ we denote by ${{}^{\mathrm{t}}\!\, } \mathfrak{g}$ the linear transpose of $\mathfrak{g}$. The involution 
\[
\mathfrak{g}\mapsto \mathfrak{g}^{\vee}=({}^{\mathrm{t}}\!\, \mathfrak{g})^{-1}
\]
determines an exterior and algebraic automorphism of the group $\mathrm{Aut}(\mathbb{P}^n_\mathbb{C})$ (\emph{see} \cite{Dieudonne}). 

\begin{thm}[\cite{Cantat}]\label{thm:cantat}
{\sl Let $M$ be a smooth, connected, complex projective variety, and let $n$ be its dimension. Let $k$ be a positive integer, and let $\rho\colon\mathrm{Aut}(\mathbb{P}^k_\mathbb{C})\to\mathrm{Bir}(M)$ be an injective morphism of groups. Then $n\geq k$, and if $n=k$ there exists a field morphism $\kappa\colon\mathbb{C}\to\mathbb{C}$, and a birational map~$\psi\colon M\dashrightarrow\mathbb{P}^n_\mathbb{C}$ such that either
\[
\psi\,\rho(\mathfrak{g})\,\psi^{-1}={}^{\kappa}\!\, \mathfrak{g}\qquad\forall\,\mathfrak{g}\in\mathrm{Aut}(\mathbb{P}^n_\mathbb{C})
\]
or
\[
\psi\,\rho(\mathfrak{g})\,\psi^{-1}=({}^{\kappa}\!\, \mathfrak{g})^{\vee}\qquad\forall\,\mathfrak{g}\in\mathrm{Aut}(\mathbb{P}^n_\mathbb{C});
\]
in particular $M$ is rational. Moreover, $\kappa$ is an automorphism of $\mathbb{C}$ if $\rho$ is an isomorphism.}
\end{thm}

Let us give the proof of Theorem \ref{thm:ptesgpes2}:

\begin{thm}\label{thm:autbirpn}
{\sl Let $\varphi$ be an automorphism of $\mathrm{Bir}(\mathbb{P}^n_\mathbb{C})$. There exists a birational map $\psi$ of $\mathbb{P}^n_\mathbb{C}$, and a field automorphism $\kappa$ such that 
\[
\varphi(g)={}^{\kappa}\!\, (\psi g\psi^{-1})\qquad \forall\,g\in G_n(\mathbb{C}).
\]} 
\end{thm}

\begin{proof}
Let us consider $\varphi\in\mathrm{Aut}(\mathrm{Bir}(\mathbb{P}^n_\mathbb{C}))$. Theorem \ref{thm:cantat} implies that up to birational conjugacy and up the action of a field automorphism 
\begin{equation}\label{eq:alternative}
\left\{
\begin{array}{ll}
\text{ either $\varphi(\mathfrak{g})=\mathfrak{g}\qquad\forall\,\mathfrak{g}\in\mathrm{Aut}(\mathbb{P}^n_\mathbb{C})$}\\
\text{or $\varphi(\mathfrak{g})=\mathfrak{g}^{\vee}\qquad\forall\,\mathfrak{g}\in\mathrm{Aut}(\mathbb{P}^n_\mathbb{C}).$}
\end{array}
\right.
\end{equation}
In other words up to birational conjugacy and up to the action of a field automorphism one cas assume that either $\varphi_{\vert\mathrm{Aut}(\mathbb{P}^n_\mathbb{C})}\colon \mathfrak{g}\mapsto\mathfrak{g}$, or $\varphi_{\vert\mathrm{Aut}(\mathbb{P}^n_\mathbb{C})}\colon \mathfrak{g}\mapsto\mathfrak{g}^\vee$.
Now determine $\varphi(\sigma_n)$. Let us work in the affine chart $z_n=1$. For $0\leq i\leq n-2$ denote by $\tau_i$ the automorphism of $\mathbb{P}^n_\mathbb{C}$ that permutes $z_i$ and $z_{n-1}$
\[
\tau_i=\big(z_0,z_1,\ldots,z_{i-1},z_{n-1},z_{i+1},z_{i+2},\ldots,z_{n-2},z_i\big).
\]
Let $\eta$ be given by
\[
\eta=\left(z_0,z_1,\ldots,z_{n-2},\frac{1}{z_{n-1}}\right).
\]
One has 
\[
\sigma_n=\big(\tau_0\eta\tau_0\big)\,\big(\tau_1\eta\tau_1\big)\,\ldots\,\big(\tau_{n-2}\eta\tau_{n-2}\big)\,\eta
\]
so
\[
\varphi(\sigma_n)=\big(\varphi(\tau_0)\varphi(\eta)\varphi(\tau_0)\big)\big(\varphi(\tau_1)\varphi(\eta)\varphi(\tau_1)\big)\ldots\big(\varphi(\tau_{n-2})\varphi(\eta)\varphi(\tau_{n-2})\big)\varphi(\eta).
\] 
Since any $\tau_i$ belongs to $\mathrm{Aut}(\mathbb{P}^n_\mathbb{C})$ one can, thanks to (\ref{eq:alternative}), compute $\varphi(\tau_i)$, and one gets: $\varphi(\tau_i)=\tau_i$. 

\smallskip

Let us now focus on $\varphi(\eta)$. We will distinguish the two cases of (\ref{eq:alternative}). Assume that $\varphi_{\vert\mathrm{PGL}(n+1;\mathbb{C})}=\mathrm{id}$. For any $\alpha=(\alpha_0,\alpha_1,\ldots,\alpha_{n-1})$ in $(\mathbb{C}^*)^n$ set 
\[
\mathfrak{d}_\alpha=(\alpha_0z_0,\alpha_1z_1,\ldots,\alpha_{n-1}z_{n-1}); 
\]
the involution $\eta$ satisfies for any $\alpha=(\alpha_0,\alpha_1,\ldots,\alpha_{n-1})\in (\mathbb{C}^*)^n$
\[
\mathfrak{d}_\beta\eta=\eta \mathfrak{d}_\alpha
\]
where $\beta=(\alpha_0,\alpha_1,\ldots,\alpha_{n-1}^{-1})$. Hence $\varphi(\eta)=\left(\pm z_0,\pm z_1,\ldots,\pm z_{n-2},\frac{\alpha}{z_{n-1}}\right)$ for $\alpha\in\mathbb{C}^*$. As $\eta$ commutes with
\[
\mathfrak{t}=\big(z_0+1,z_1+1,\ldots,z_{n-2}+1,z_{n-1}\big), 
\]
the image $\varphi(\eta)$ of $\eta$ commutes to $\varphi(\mathfrak{t})=\mathfrak{t}$. Therefore 
\[
\varphi(\eta)=\left(z_0,z_1,\ldots,z_{n-2},\frac{\alpha}{z_{n-1}}\right).
\]
If 
\[
\mathfrak{h}_n=\left(\frac{z_0}{z_0-1},\frac{z_0-z_1}{z_0-1},\frac{z_0-z_2}{z_0-1},\ldots,\frac{z_0-z_{n-1}}{z_0-1}\right)
\]
then $\varphi(\mathfrak{h}_n)=\mathfrak{h}_n$, and $(\mathfrak{h}_n\sigma_n)^3=\mathrm{id}$ implies that $\varphi(\sigma_n)=~\sigma_n$. If $\varphi_{\vert\mathrm{PGL}(n+1;\mathbb{C})}$ coincides with $\mathfrak{g}\mapsto \mathfrak{g}^{\vee}$, a similar argument yields $\big(\varphi(\mathfrak{h}_n)\varphi(\sigma_n)\big)^3\not=\mathrm{id}$.
\end{proof}

\subsection{Simplicity of $G_n(\mathbb{C})$}

An \textbf{\textit{algebraic family}} of $\mathrm{Bir}(\mathbb{P}^n_\mathbb{C})$ is the data of a rational map 
\[
\phi\colon M\times\mathbb{P}^n_\mathbb{C}\dashrightarrow\mathbb{P}^n_\mathbb{C}, 
\]
where $M$ is a $\mathbb{C}$-variety, defined on a dense open subset $\mathcal{U}$ such that
\begin{itemize}
\item for any $m\in M$ the intersection $\mathcal{U}_m=\mathcal{U}\cap(\{m\}\times\mathbb{P}^n_\mathbb{C})$ is a dense open subset of $\{m\}\times\mathbb{P}^n_\mathbb{C}$,

\item and the restriction of $\mathrm{id}\times \phi$ to $\mathcal{U}$ is an isomorphism of $\mathcal{U}$ on a dense open subset of $M\times\mathbb{P}^n_\mathbb{C}$.
\end{itemize}

For any $m\in M$ the birational map $z\dashrightarrow \phi(m,z)$ represents an element $\phi_m$ in $\mathrm{Bir}(\mathbb{P}^n_\mathbb{C})$; the map 
\[
M\to\mathrm{Bir}(\mathbb{P}^n_\mathbb{C}),\qquad m\mapsto \phi_m
\]
is called \textbf{\textit{morphism}} from $M$ to $\mathrm{Bir}(\mathbb{P}^n_\mathbb{C})$. These notions yield the natural Zariski topology on~$\mathrm{Bir}(\mathbb{P}^n_\mathbb{C})$, introduced by Demazure (\cite{Demazure}) and Serre (\cite{Serre}): the subset $\Omega$ of~$\mathrm{Bir}(\mathbb{P}^n_\mathbb{C})$ is \textbf{\textit{closed}} if for any $\mathbb{C}$-variety $M$, and any morphism $M\to\mathrm{Bir}(\mathbb{P}^n_\mathbb{C})$ the preimage of $\Omega$ in $M$ is closed. Note that in restriction to $\mathrm{Aut}(\mathbb{P}^n_\mathbb{C})$ one obtains the usual Zariski topology of the algebraic group $\mathrm{Aut}(\mathbb{P}^n_\mathbb{C})=\mathrm{PGL}(n+1;\mathbb{C})$.

Let us recall the following statement: 

\begin{pro}[\cite{Blanc}]\label{pro:blanc}
{\sl Let $n\geq 2$. Let $\mathrm{H}$ be a non-trivial, normal, and closed subgroup of~$\mathrm{Bir}(\mathbb{P}^n_\mathbb{C})$. Then $\mathrm{H}$ contains $\mathrm{Aut}(\mathbb{P}^n_\mathbb{C})$ and $\mathrm{PSL}\big(2;\mathbb{C}(z_0,z_1,\ldots,z_{n-2})\big)$.}
\end{pro}

In our context we have a similar statement:

\begin{pro}\label{pro:chouetteinclusion}
{\sl Let $n\geq 2$. Let $\mathrm{H}$ be a non-trivial, normal, and closed subgroup of~$G_n(\mathbb{C})$. Then $\mathrm{H}$ contains $\mathrm{Aut}(\mathbb{P}^n_\mathbb{C})$ and $\sigma_n$.}
\end{pro}

\begin{proof}
A similar argument as in \cite{Blanc} allows us to prove that $\mathrm{Aut}(\mathbb{P}^n_\mathbb{C})$ is contained in $\mathrm{H}$.

The fact $-\mathrm{id}$ and $\sigma_n$ are conjugate in $G_n(\mathbb{C})$ (\emph{see} Proof of Proposition \ref{pro:nonlin2}) yields the conclusion.
\end{proof}

The proof of Proposition \ref{thm:ptesgpes3} follows from Proposition \ref{pro:chouetteinclusion} and Corollary \ref{cor:prodconj}.

\bibliographystyle{plain}

\bibliography{biblio}

\end{document}